\documentclass[11pt,a4paper]{article}
\usepackage{amsmath}
\usepackage{amsthm}
\usepackage{amssymb}
\usepackage{amscd}
\usepackage{graphicx}
\usepackage{epsfig}
\usepackage[matrix,arrow,curve]{xy}
\usepackage{color}

\usepackage{bbm}
\usepackage{stmaryrd}
\usepackage{hyperref}
\usepackage[usenames,dvipsnames]{xcolor}

\usepackage{latexsym}
\usepackage{amsfonts}
\input xy
\usepackage{tikz}
\usetikzlibrary{arrows,chains,matrix,positioning,scopes,snakes}
\makeatletter
\tikzset{join/.code=\tikzset{after node path={%
\ifx\tikzchainprevious\pgfutil@empty\else(\tikzchainprevious)%
edge[every join]#1(\tikzchaincurrent)\fi}}}
\makeatother

\tikzset{>=stealth',every on chain/.append style={join},
         every join/.style={->}}
\tikzset{
    >=stealth',
    punkt/.style={
           rectangle,
           rounded corners,
           draw=black, very thick,
           text width=6.5em,
           minimum height=2em,
           text centered},
    pil/.style={
           ->,
           thick,
           shorten <=2pt,
           shorten >=2pt,}
}

\xyoption{all} \tolerance=500

\newtheorem{thm}{Theorem}[section]
\newtheorem{theorem}{Theorem}[section]

\newtheorem{corollary}[thm]{Corollary}

\theoremstyle{definition}
\newtheorem{example}[thm]{Example}

\newtheorem{definition}[thm]{Definition}

\font\black=cmbx10 \font\sblack=cmbx7 \font\ssblack=cmbx5 \font\blackital=cmmib10  \skewchar\blackital='177
\font\sblackital=cmmib7 \skewchar\sblackital='177 \font\ssblackital=cmmib5 \skewchar\ssblackital='177
\font\sanss=cmss12 \font\ssanss=cmss8 scaled 900 \font\sssanss=cmss8 scaled 600 \font\blackboard=msbm10
\font\sblackboard=msbm7 \font\ssblackboard=msbm5 \font\caligr=eusm10 \font\scaligr=eusm7 \font\sscaligr=eusm5

\font\bsymb=cmsy10 scaled\magstep2
\def\all#1{\setbox0=\hbox{\lower1.5pt\hbox{\bsymb
       \char"38}}\setbox1=\hbox{$_{#1}$} \box0\lower2pt\box1\;}
\def\exi#1{\setbox0=\hbox{\lower1.5pt\hbox{\bsymb \char"39}}
       \setbox1=\hbox{$_{#1}$} \box0\lower2pt\box1\;}

\def\tx#1{{\fam0\relax#1}}

\newfam\bifam
\textfont\bifam=\blackital \scriptfont\bifam=\sblackital \scriptscriptfont\bifam=\ssblackital

\newfam\blfam
\textfont\blfam=\black \scriptfont\blfam=\sblack \scriptscriptfont\blfam=\ssblack

\newfam\bbfam
\textfont\bbfam=\blackboard \scriptfont\bbfam=\sblackboard \scriptscriptfont\bbfam=\ssblackboard

\newfam\ssfam
\textfont\ssfam=\sanss \scriptfont\ssfam=\ssanss \scriptscriptfont\ssfam=\sssanss
\def\sss#1{{\fam\ssfam\relax#1}}

\newfam\clfam
\textfont\clfam=\caligr \scriptfont\clfam=\scaligr \scriptscriptfont\clfam=\sscaligr

\def\pmb#1{\setbox0\hbox{${#1}$} \copy0 \kern-\wd0 \kern.2pt \box0}
\def\pmbb#1{\setbox0\hbox{${#1}$} \copy0 \kern-\wd0
      \kern.2pt \copy0 \kern-\wd0 \kern.2pt \box0}
\def\pmbbb#1{\setbox0\hbox{${#1}$} \copy0 \kern-\wd0
      \kern.2pt \copy0 \kern-\wd0 \kern.2pt
    \copy0 \kern-\wd0 \kern.2pt \box0}
\def\pmxb#1{\setbox0\hbox{${#1}$} \copy0 \kern-\wd0
      \kern.2pt \copy0 \kern-\wd0 \kern.2pt
      \copy0 \kern-\wd0 \kern.2pt \copy0 \kern-\wd0 \kern.2pt \box0}
\def\pmxbb#1{\setbox0\hbox{${#1}$} \copy0 \kern-\wd0 \kern.2pt
      \copy0 \kern-\wd0 \kern.2pt
      \copy0 \kern-\wd0 \kern.2pt \copy0 \kern-\wd0 \kern.2pt
      \copy0 \kern-\wd0 \kern.2pt \box0}

\mathchardef\za="710B  
\mathchardef\zb="710C  
\mathchardef\zg="710D  
\mathchardef\zd="710E  
\mathchardef\zve="710F 
\mathchardef\zz="7110  
\mathchardef\zh="7111  
\mathchardef\zvy="7112 
\mathchardef\zi="7113  
\mathchardef\zk="7114  
\mathchardef\zl="7115  
\mathchardef\zm="7116  
\mathchardef\zn="7117  
\mathchardef\zx="7118  
\mathchardef\zp="7119  
\mathchardef\zr="711A  
\mathchardef\zs="711B  
\mathchardef\zt="711C  
\mathchardef\zu="711D  
\mathchardef\zvf="711E 
\mathchardef\zq="711F  
\mathchardef\zc="7120  
\mathchardef\zw="7121  
\mathchardef\ze="7122  
\mathchardef\zy="7123  
\mathchardef\zf="7124  
\mathchardef\zvr="7125 
\mathchardef\zvs="7126 
\mathchardef\zf="7127  
\mathchardef\zG="7000  
\mathchardef\zD="7001  
\mathchardef\zY="7002  
\mathchardef\zL="7003  
\mathchardef\zX="7004  
\mathchardef\zP="7005  
\mathchardef\zS="7006  
\mathchardef\zU="7007  
\mathchardef\zF="7008  
\mathchardef\zW="700A  

\newcommand{\be}{\begin{equation}}
\newcommand{\ee}{\end{equation}}

\newcommand{\bea}{\begin{eqnarray}}
\newcommand{\eea}{\end{eqnarray}}
\newcommand{\beas}{\begin{eqnarray*}}
\newcommand{\eeas}{\end{eqnarray*}}
\def\*{{\textstyle *}}

\newcommand{\ot}{\otimes}

\newcommand{\pa}{\partial}

\newcommand{\ad}{{\rm ad}}

\def\ran{\rangle}

\def\sT{{\sss T}}

\def\sw{{\sss w}}

\def\xd{\tx{d}}

\newdir{|>}{%
!/4.5pt/@{|}*:(1,-.2)@^{>}*:(1,+.2)@_{>}}

\newcommand{\p}{\partial}
\newcommand{\la}{\langle}

\newcommand{\Z}{\mathbb{Z}}
\newcommand{\R}{\mathbb{R}}

\def\deg{{\sss deg}}

\def\cm{\color{magenta}}

\def\nr{\mathfrak{nr}}

\newcommand{\wu}{\operatorname{deg}}

\begin{document}
\title{\bf Transitive nilpotent Lie algebras of vector fields\\ and their Tanaka prolongations \thanks{Research  founded by the  Polish National Science Centre grant
under the contract number 2016/22/M/ST1/00542.
 }}
\date{}
\author{\\ Katarzyna  Grabowska$^1$\\ Janusz Grabowski$^2$\\ Zohreh Ravanpak$^2$
        \\ \\
         $^1$ {\it Faculty of Physics}\\
                {\it University of Warsaw}\\
                \\$^2$ {\it Institute of Mathematics}\\
                {\it Polish Academy of Sciences}
                }
\maketitle

\begin{abstract}{Transitive local Lie algebras of vector fields can be easily constructed from dilations of $\R^n$ associating with variables positive weights (give me a sequence of $n$ positive integers and I will give you a transitive nilpotent Lie algebra of vector fields on $\R^n$). It is interesting that all transitive nilpotent local Lie algebra of vector fields can be obtained as subalgebras of nilpotent algebras obtained in this way. Starting with a graded nilpotent Lie algebra
one constructs graded parts of its Tanaka prolongations  inductively as `derivations of degree 0,1,etc. Of course, vector fields of weight $k$ with respect to the dilation define automatically derivations of weight $k$, so the Tanaka prolongation is in this case never finite.  Are they all such derivations given by vector fields or there are additional `strange'ones? We answer this question.
 Except for special cases, derivations of degree 0 are given by vector fields of degree 0 and the Tanaka prolongation recovers the whole algebra of polynomial vectors defined by the dilation. However, in some particular cases of dilations we can find `strange' derivations which we describe in detail.}
\end{abstract}
\def\cm{\color{magenta}}

\vspace{2mm} \noindent {\bf MSC 2010}: (primary)  17B30  17B66
 (secondary)  57R25 57S20.

\section{Introduction}

Origins of the theory of Lie groups go back to 1870s when Sophus Lie studied objects called at that time continuous groups \cite{Lie2,Lie3}.  To be precise, these objects were not groups in modern sense of this term. They were rather locally defined families of diffeomorphisms, closed with respect to the operations of composition and inverse, whenever they were correctly defined. Nowadays such structures are called pseudogroups.

Special examples of pseudogroups studied by Lie were ``finite continuous groups'' i.e. representations of abstract finite dimensional Lie groups in diffeomorphisms on manifolds. The theory of Lie groups is now a standard and classical part of mathematical education, including also the theory of Lie algebras and their representations.

An important aspect of Lie theory is the problem of classification of Lie algebras. It started by the conjecture by Killing and Cartan that any finite-dimensional real Lie algebra $\mathfrak{g}$ is the semidirect product of a solvable ideal and a semisimple subalgebra. The conjecture was later proved by Eugenio Elia Levi in 1905 and is now known as \emph{Levi decomposition}. The first component is the maximal solvable ideal i.e. the \emph{radical} of the algebra. The semisimple subalgebra being the second component is called a \emph{Levi subalgebra}. Since any finite dimensional Lie algebra is a a semidirect product of a solvable Lie algebra and a semisimple Lie algebra we can classify both classes of algebras separately.

Semisimple complex Lie algebras have been completely classified by E. Cartan \cite{Cartan1894}, while the  real case was solved by F. Gantmacher \cite{Gantmacher39}. By definition every semisimple Lie algebra over an algebraically closed field of characteristic 0 is a direct sum of simple Lie algebras. Every finite-dimensional simple Lie algebras belongs to one of the four families: $A_n, B_n, C_n$, and $D_n$, or is isomorphic to one of the exceptional Lie algebras denoted usually by $E_6, E_7, E_8, F_4$, and $G_2$. Simple Lie algebras correspond to connected Dynkin diagrams. Semisimple Lie algebras can be classified by Dynkin diagrams that are not necessarily connected: each connected component of the diagram corresponds to a simple component in the decomposition of a semisimple Lie algebra into simple Lie algebras.

Calssification of solvable components of the Levi decomposition has proven to be much more complicated -- it is now  unsolved and moreover believed to be unsolvable in full generality, i.e. for arbitrarily large dimension. Classification of nilpotent Lie algebras is known up to dimension eight \cite{10,11} while the solvable case terminates in dimension six.

In this paper we discuss realisations of Lie algebras in vector fields. They are of special interest due to a wide range of applications in various aspects of differential geometry, e.g the general theory of differential equations and their systems including classification problems for ODEs and PDEs, integration of differential equations, the theory of systems with superposition principles, geometric control theory, to name just a few. The literature on classification of local transitive realisations of Lie algebras in vector fields is already quite extensive. Theoretical results as well as powerful computational methods can be found e.g in \cite{Blattner69,Draisma12,GS64,SW014}. A classification of Lie algebras of vector fields on the real plane is described in \cite{92}. The solvable Lie algebras of vector fields are discussed thoroughly in survey \cite{SW014} and the references therein. Transitive solvable Lie algebras of vector fields have been also studied in the context of integrability by quadratures \cite{CFGR15,CFG16}. Note that finite-dimensional solvable Lie algebras of vector fields are not far from the nilpotent ones, since they are characterized as those whose derived ideal is nilpotent.

We will be interested in the graded nilpotent Lie algebras of vector fields which have negative weights with respect to a dilation. This is not very restrictive assumption since all transitive nilpotent Lie algebras of vector fields are subalgebras of the algebras of this type \cite{JG90}. Of course, we can consider the infinite dimensional algebra spanned by all homogeneous vector fields of the dilation. Another graded extension of the nilpotent algebra is the Tanaka prolongation \cite{Tan1,Tan2}.
Our conjecture was that both prolongations coincide. However, unexpectedly it turned out that for some particular dilations there are `strange' derivations which not come form vector fields.
In this paper we analyse the situation in detail.

\section{Nilpotent Lie algebras of vector fields}\label{secnil}
Probably the easiest way of constructing a nilpotent Lie algebra of vector fields is to consider vector fields with negative weight with respect to a positive dilation on $\R^n$. To be more precise, let us associate with the coordinates $(x^1,\dots,x^n)$ positive integer degrees (weights) $(r_1, r_2,\ldots, r_n)$. This is equivalent to picking up the \emph{weight vector field}
$$\nabla^h=\sum_ir_ix^i\pa_{x^i}\,,$$
or a one-parametric family of positive \emph{dilations} (called in \cite{Bruce:16, Grabowski:2012}, a \emph{homogeneity structure})
$$h_t(x^1,\dots,x^n)=\left(t^{r_1}x^1,\dots,t^{r_n}x^n\right)\,.$$
The biggest $r_i$ is called the \emph{weight} of the homogeneity structure and denoted $\sw(h)$. Permuting variables, any dilation is determined by its  \emph{signature} $(r_1\leq r_2\leq r_3\leq...\leq r_n)$. We alweys will use the above ordered signature. In this sense $\sw(h)=r_n$. The weight vector field defines \emph{homogeneous functions of weight $k$} as those $f\in C^\infty(\R^n)$ for which $\nabla^h(f)=kf$. One can prove \cite{Grabowski:2006,Grabowski:2012} that only non-negative homogeneity degrees are allowed and that each homogeneous function is polynomial in $(x^1,\dots,x^n)$. Thus we have simultaneously two degrees: the degree of a homogeneous function as the polynomial in variables $x_i$, which we will denote by $\wu$, and the weight with respect to the dilation, denoted $\sw(x_i)$ or just $r_i$. For instance $x_n$ is first degree polynomial with weight $r_n$.  This can be extended to tensors, e.g. a vector field $X$ is homogeneous of degree $k$ (what we denote by $X\in\mathfrak{g}_k(h)$) if
$[\nabla^h,X]=kX$. This time, $k$ may be negative, for instance $\sw(\pa_{x^i})=-r_i$ while $\sw(x^j\pa_{x^i})=r_j-r_i$. For homogenous $X,Y$,
$$\sw([X,Y])=\sw(X)+\sw(Y)\,,$$
that implies trivially that the family of vector fields of negative degrees
$$\mathfrak{g}_{<0}(h)=\bigoplus_{i=-\wu(h)}^{-1}\mathfrak{g}_i(h)$$
is a (finite-dimensional) nilpotent Lie algebra (cf. \cite{JG90,MK88}). This (except for finite-dimensionality) remains valid for general graded supermanifolds \cite{Voronov10}, in particular graded bundles in the sense of \cite{Grabowski:2012}. Of course, for a given dilation the whole graded Lie algebra
$$\mathfrak{g}_{\infty}(h)=\mathfrak{g}_{<0}(h)\oplus\mathfrak{g}_{0}(h)\oplus\mathfrak{g}_{1}(h)\oplus\mathfrak{g}_{2}(h)\oplus...
$$
is defined. Interestingly enough the total graded algebra is determined (with some exceptions) by means of the so called \emph{Tanaka prolongation} \cite{Tan1,Tan2}.

On the other hand, a strong result \cite{JG90} states that any transitive nilpotent Lie algebra $L$ of vector fields defined on a neighbourhood of $0\in\R^n$ is locally a subalgebra of $\mathfrak{g}_{<0}(h)$ for some dilation $h$ on $\R^n$ (cf. also \cite{MK88}, where the assumptions are a little bit stronger). The nilpotent algebra $L$ is therefore automatically finite-dimensional and polynomial in appropriate homogeneous coordinates $(x^1,\dots,x^n)$. A similar result is valid for solvable Lie algebras of vector fields \cite{Grabowska:2019}. The homogeneity structure is uniquely defined by the structure of $L$, so that $\mathfrak{g}_{<0}(h)$ can be viewed as nilpotent `enveloping' of $L$. Note also that every finite-dimensional (real) Lie algebra is a transitive Lie algebra of vector fields as the Lie algebra of left-invariant vector fields on the corresponding 1-connected Lie group.

The homogeneity structures defined in \cite{Grabowski:2012} are a little bit more general: they admit coordinates of degree $0$, so that the manifold is a fibration over the quotient manifold corresponding to zero-degree coordinates (graded bundle). In particular, graded bundles of degree 1 are just vector bundles. The family of vector fields $\mathfrak{g}_{<0}(h)$ is still nilpotent and it is a subalgebra in the family of vector fields with non-positive degrees,
$$\mathfrak{g}_{\le 0}(h)=\bigoplus_{i=-\wu(h)}^{0}\mathfrak{g}_i(h)\,.$$
The latter is generally not solvable, but any finite-dimensional Lie algebra $L$ in $\mathfrak{g}_{\le 0}(h)$ such that $[L,L]\subset\mathfrak{g}_{< 0}(h)$ is solvable \cite{Grabowska:2019}. This is because, for finite-dimensional Lie algebras, $L$ is solvable if and only if $[L,L]$ is nilpotent (this is not true in infinite dimensions). In other words the \emph{nilradical} $\nr(L)$ of $L$ contains $[L,L]$. Solvable Lie algebras of vector fields in $\mathfrak{g}_{\le 0}(h)$ we will call \emph{dilational} (with respect to $h$) (cf. \cite{Grabowska:2019}).

The method of defining dilational transitive solvable Lie algebras of vector fields is illustrated by the following example.
\begin{example}\label{1.1}
On $\R^3$ with coordinates $(x,y,z)$ we define the dilation $h$ declaring that $x$ is of weight 1, $y$ is of weight 2, $z$ is of weight 3. The nilpotent part $\mathfrak{g}_{< 0}(h)$ is a vector space over $\R$ generated by polynomial vector fields
\be \mathfrak{g}_{-3}(h)=\R[\,\pa_z\,],\, \mathfrak{g}_{-2}(h)=\R[\,x\pa_z,\,\pa_y\,]\,\ \
\mathfrak{g}^{-1}(h)=\R[\, \pa_x,\, y\pa_z]\,.
\ee
\end{example}
As polynomial vector fields have an additional grading related to the degree of polynomial (linear, quadratic, etc), so we will use both gradings to read the necessary information about the Lie algebras of transitive vector fields.

 \subsection{Tanaka prolongation of nilpotent algebra associated with a dilation}

The following is well known. Let L be a Lie algebra. We define the lower central series of L inductively:
$$L=L^1\,,\ L^{i+1} = [L,L^i ] \ \text{for}\ i\ge 1\,.$$
\begin{definition} If the lower central series \emph{stops at zero}, i.e. $L^{k+1}=\{ 0\}$ for some $k$, then we say that $L$ is \emph{nilpotent}.
The smallest such $k$ we call the \emph{height} of the nilpotent Lie algebra.
\end{definition}
Let us recall from section \ref{secnil} that on $\R^n$ we consider a positive dilation, i.e. an action $h$ of the multiplicative monoid $(\R,\cdot)$
of the form
$$h_t(x^1,\dots,x^n)=\left(t^{r_1}x^1,\dots,t^{r_n}x^n\right)\,,\ t\in\R\,,$$
where $r_i\in\Z$, $r_i>0$, $\sw(h)=r_n=\max\{r_1\le \ldots\le r_n\}.$ The associated \emph{weight vector field} reads $\nabla^h=\sum_ir_ix^i\pa_{x^i}$.
A smooth function $f$ defined in a neighbourhood of $0\in\R^n$ is homogeneous of weight $a$ if $\nabla^h(f)=af$. Note that only non-negative integer homogeneity degrees are possible, and any such function admits a global polynomial form. Similarly, a vector field $X$ is homogeneous of weight $a$ if
$[\nabla^h,X]=aX$, however now negative weights not less than $-\sw(h)$ are allowed. The family of homogeneous vector fields (of weight $a$) is denoted by $\mathfrak{g}_a(h)$.

We observe that the family of vector fields of negative weights
$$\mathfrak{m}=\mathfrak{g}_{<0}(h)=\bigoplus_{i=-\sw(h)}^{-1}\mathfrak{g}_i(h)$$
is a transitive (finite-dimensional) graded nilpotent Lie algebra of vector fields (cf. \cite{JG90,MK88}). The transitivity means that the vector fields of negative weights span $\sT\R^n$. Of course, any transitive subalgebra of $\mathfrak{g}_{<0}(h)$ is again transitive nilpotent Lie algebra of vector fields.

It can be checked that the product of two homogeneous functions of weight $v_1$ and $v_2$ is homogeneous of weight $v_1+v_2$. Consequently for any multi-index $\alpha=(\alpha_1,...,\alpha_n)\in{\mathbb N}^{n} $, the function
$$x^{\alpha }:=x_1^{\alpha_1}...x_n^{\alpha_n}$$
are homogeneous of weight $r_1\alpha_1+...+r_n\alpha_n$.

Let us fix a homogeneity structure on $\mathbb R^n$, set $w=\sw(h)=r_n$, and let denote the space of vector fields generated by homogeneous vector fields of negative degree by ${\mathfrak g}_{<0}(h)$. The algebra generated by this vector fields is a nilpotent algebra $$\mathfrak m={\mathfrak g}_{-w}(h)\oplus...\oplus{\mathfrak g}_{-1}(h)\,.$$

The $l$th \emph{Tanaka prolongation} is given by
$${\mathfrak g}^{(l)}_T({\mathfrak m})={\mathfrak m}\oplus{\mathfrak{g}_0^T}\oplus{\mathfrak{g}_1^T}\oplus...\oplus{\mathfrak{g}_l^T}\,
$$
where spaces ${\mathfrak{g}_k^T}$  for $g\geq 0$ are defined inductively
$${\mathfrak g}_k^T=\left\{ v\in  \bigoplus_{p<0}{\mathfrak g}_{p+k}^T\ot{\mathfrak g}_p(h)^*\ |\ v[X,Y]=[v(X),Y]+[x,v(Y)]\right\}\,.$$
In the above all spaces $\mathfrak{g}_i^T$ for $i<0$ equal $\mathfrak{g}_i(h)$ by definition while $\mathfrak{g}_i^T$ for $i\geq 0$ may in principle be bigger than $\mathfrak{g}_i(h)$. It is clear that vector field of weight $i$ defines the derivation of weight $i$.

The space ${\mathfrak g}_k^T$ for $k\ge 0$ will be called, with some abuse of terminology, the \emph{space of graded derivations of ${\mathfrak m}$ of weight $k$}.
It is clear that every vector field $X \in \mathbb R^n$ of weight $p$ defines a derivation $D_X$ such that $D_X(Y) = [X, Y]$. So, if $Y$ is a vector field of weight $i<0$ then $D_X(Y)$ is a vector field of weight $p+i$. Therefore, $[{\mathfrak g}_p^T, {\mathfrak g}_{-1}(h)]\subset\mathfrak g_{p-1}^T$. However, we will not assume that ${\mathfrak m}$ is fundamental, i.e. that $[{\mathfrak g}_p(h), {\mathfrak g}_{-1}(h)]={\mathfrak g}_{p-1}(h)$. In particular, if the nonzero weights are only even (the signature consists of even numbers), then ${\mathfrak g}_{-1}(h)=\{ 0\}$.

\medskip 	
Thus we will work with the graded nilpotent Lie algebra  ${\mathfrak m}={\mathfrak g}_{-w}(h)\oplus\dots\oplus{\mathfrak g}_{-1}(h)$, trying to find its Tanaka prolongations. We want to find the condition that the dilation $h$ is such that ${\mathfrak g}_{s}^T={\mathfrak g}_{s}(h)$  for $s\le l$, i.e. the derivations of weight $s$  come from vector fields of weight $s$. Such derivations we will call \emph{of the first kind}.

That this need not be the case in general shows the following.
\begin{theorem}
Let $h$ be a dilation on $\R^n$ such that
\be\label{crucial}\ell=r_n-2r_{n-1}\ge 0\,.
\ee
Let $A=(a_i^j)_{i,j=1}^{n-1}$ be a symmetric matrix of real numbers such that $a_i^j\ne 0$ only for $r_i=r_j=r_{n-1}$.

Then there is a unique derivation of weight $\ell$,
$$D_{\ell}(A)\in  \bigoplus_{p<0}{\mathfrak g}_{p+\ell}(h)\ot{\mathfrak g}^*_p(h)$$
which vanishes on constant vector fields and on linear vector fields reads
\be\label{dell}D_{\ell}^A(x_i\pa_j)=\zd_j^na_i^k\pa_k\,.\ee
As vanishing on constant vector fields and being of weight $\ell\geq 0$, this derivation does not come from a vector field of weight $\ell$.
The derivations $D_{\ell}^A$ we will call \emph{of the second kind}.
\end{theorem}

\begin{proof}
It is clear that there is exactly one extension of $D_{\ell}^A$ defined as above on constant and linear vector fields to
$$D_\ell^A\in  \bigoplus_{p<0}{\mathfrak g}_{p+\ell}(h)\ot{\mathfrak g}^*_p(h)$$
such that for polynomials $f,g$ we have another derivation property:
\be\label{dp}D_\ell^A(fg\pa_j)=fD_\ell^A(g\pa_j)+gD_\ell^A(f\pa_j)\,.\ee
We will show that $D_\ell^A$ is a derivation of the Lie algebra ${\mathfrak m}={\mathfrak g}_{< 0}(h)$ of weight $\ell$.

First let us consider the weight of $D_\ell^A$. Note that $D_\ell^A$ is of weight $\ell$ on constant and linear vector fields.

 Indeed, from (\ref{dell})
we read the weight of $D_{\ell}^A(x_i\pa_j)$ is $-r_{n-1}$. Since $D_{\ell}^A(x_i\pa_j)$ is nonzero only for $j=n$ and $r_i=r_{n-1}$ we calculate the weight of $D_\ell^A$ as $-r_{n-1}-(r_{n-1}-r_n)=r_n-2r_{n-1}=\ell$. Then inductively with the use of (\ref{dp}) we prove that $D_\ell^A$ is of weight $\ell$ on higher polynomial degree vector fields. The weight of $x_if\pa_j$ is $r_i+\sw(f)-r_j$. Acting with $D_\ell^A$ we get
$$D_\ell^A(x_if\pa_j)=fD_\ell^A(x_i\pa_j)+x_iD_\ell^A(f\pa_j).$$
The weight of first summand is $\sw(f)+r_i-r_j+\ell$ by the definition of $D_\ell^A$ and the weight of the second summand is also $r_i+\sw(f)-r_j+\ell$ by the inductive assumption, we conclude then that $D_\ell^A$ is indeed of weight $\ell$.

Now we show that $D_\ell^A$ is a derivation, i.e. that for all polynomials $f$ and $g$ the Leibniz rule
\begin{equation}\label{lei}
D_\ell^A([f\pa_j, g\pa_k])=[D_\ell^A(f\pa_j), g\pa_k]+[f\pa_j, D_\ell^A(g\pa_k)]
\end{equation}
is satisfied.  Note that if $\wu(f)=\wu(g)=0$ and $\wu(f)+\wu(g)=1$ the Leibniz rule (\ref{lei}) is satisfied trivially since $D_\ell^A$ vanishes on all vector fields of the form $\pa_i$. The first nontrivial case is then $\wu(f)=1$ and $\wu(g)=1$. First let us assume that $j<n$ and $l<n$ and consider $[x_i\pa_j, x_k\pa_l]$. Since both vector fields should be of negative weight we assume also that $r_i<r_j$ and
$r_k<r_l$. We have
$$D_\ell^A([x_i\pa_j, x_k\pa_l])=\delta_{jk}D_\ell^A(x_i\pa_l)-\delta_{li}D_\ell^A(x_k\pa_j)=0\,,$$
since $D_\ell^A$ vanishes on all linear vector fields $x_i\partial_j$ with $j<n$. For the same reason
$$[D_\ell^A(x_i\pa_j), x_k\pa_l]+[x_i\pa_j, D_\ell^A(x_k\pa_l)]=0,$$
Leibniz rule is then satisfied. Next we consider $[x_i\pa_n, x_j\pa_k]$ for $k<n$:
$$D_\ell^A([x_i\pa_n, x_j\pa_k])=\delta_{ik}D_\ell^A(-x_j\pa_n)=-\delta_{ik}a^s_j\pa_s=0\,,$$
because if $k<n$ then $r_j<r_{n-1}$ and $a^s_j=0$. On the other hand
$$[D_\ell^A(x_i\pa_n),x_j\pa_k]+[x_i\pa_n, D_\ell^A(x_j\pa_k)]=[a^s_i\pa_s,x_j\pa_k]=a^j_i\pa_k=0,$$
and again Leibniz rule is satisfied. Finally we take $[x_i\pa_n, x_j\pa_n]=0$. Since $i<n$ and $j<n$, then of course  $D_\ell^A([x_i\pa_n, x_j\pa_n])=0$ and also
$$[D_\ell^A(x_i\pa_n),x_j\pa_n]+[x_i\pa_n, D_\ell^A(x_j\pa_n)]=
[a^s_i\pa_s,x_j\pa_n]+[x_i\pa_n, a^s_j\pa_s)]=(a^j_i-a^i_j)\pa_n=0\,,$$
since $a^i_j$ is symmetric.

For the inductive step we have to show that
\begin{equation}\label{der}
D_\ell^A([x_if\pa_j,g\pa_k])=[D_\ell^A(x_if\pa_j),g\pa_k]+[x_if\pa_j,D_\ell^A(g\pa_k)]
\end{equation}
for $\deg(f)\le \za$ and  $\deg(g)\le \zb$,  $\za,\zb\ge 1$,
\begin{equation}\label{derx}
D_\ell^A([f\pa_j,g\pa_k])=[D_\ell^A(f\pa_j),g\pa_k]+[f\pa_j,D_\ell^A(g\pa_k)]
\end{equation}
for $\deg(f)\le \za$ and  $\deg(g)\le \zb$,  $\za,\zb\ge 1$.
We start from the left hand side. For the simplicity of notation we shall write $f_k=\pa_kf$
\begin{multline*}
D_\ell^A([x_if\pa_j,g\pa_k])=
D_\ell^A(x_ifg_j\pa_k-\delta_{ik}gf\pa_j-gx_if_k\pa_j)=\\
x_ifD_\ell^A(g_j\pa_k)+x_ig_jD_\ell^A(f\pa_k)+fg_jD_\ell^A(x_i\pa_k)-\delta_{ik}gD_\ell^A(f\pa_j)-\delta_{ik}fD_\ell^A(g\pa_j)- \\ x_igD_\ell^A(f_k\pa_j)-
x_if_kD_\ell^A(g\pa_j)-gf_kD_\ell^A(x_i\pa_j)\,.
\end{multline*}
The sum of all the terms containing $x_i$ equals $x_iD_\ell^A([f\pa_j, g\pa_k])$. For the left hand side of (\ref{der}) we get then
\begin{multline}\label{der1}
D_\ell^A([x_if\pa_j,g\pa_k])=\\
x_iD_\ell^A([f\pa_j, g\pa_k])+fg_jD_\ell^A(x_i\pa_k)-\delta_{ik}gD_\ell^A(f\pa_j)-\delta_{ik}fD_\ell^A(g\pa_j)-gf_kD_\ell^A(x_i\pa_j)\,.
\end{multline}
Now in the right hand side we use first (\ref{dp}) and then perform the necessary calculation
\begin{multline*}
[D_\ell^A(x_if\pa_j),g\pa_k]+[x_if\pa_j,D_\ell^A(g\pa_k)]= \\
[x_iD_\ell^A(f\pa_j)+fD_\ell^A(x_i\pa_j),g\pa_k]+x_i[f\pa_j,D_\ell^A(g\pa_k)]-fD_\ell^A(g\pa_k)(x_i)\pa_j=\\
x_i[D_\ell^A(f\pa_j), g\pa_k]-\delta_{ik}gD_\ell^A(f\partial_j)+f[D_\ell^A(x_i\pa_j), g\pa_k]-gf_kD_\ell^A(x_i\pa_j)+
\\ x_i[f\pa_j, D_\ell^A(g\pa_k)]
-fD_\ell^A(g\pa_k)(x_i)\pa_j\,.
\end{multline*}
Again the terms containing $x_i$ sum up to $x_iD_\ell^A([f\pa_j, g\pa_k])$. Finally for the right hand side we get
\begin{multline}\label{der2}
[D_\ell^A(x_if\pa_j),g\pa_k]+[x_if\pa_j,D_\ell^A(g\pa_k)]= \\
x_iD_\ell^A([f\pa_j, g\pa_k])-\delta_{ik}gD_\ell^A(f\partial_j)+f[D_\ell^A(x_i\pa_j), g\pa_k]-gf_kD_\ell^A(x_i\pa_j)
-fD_\ell^A(g\pa_k)(x_i)\pa_j\,.
\end{multline}
The difference between (\ref{der1}) and (\ref{der2}) reads then
$$f\left([D_\ell^A(x_i\pa_j), g\pa_k]-D_\ell^A(g\pa_k)(x_i)\pa_j-g_jD_\ell^A(x_i\pa_k)+\delta_{ik}D_\ell^A(g\pa_j)\right)\,.$$
The first summand can be replaced with $D_\ell^A([x_i\pa_j, g\pa_k])-[x_i\pa_j,D_\ell^A(g\pa_k)]$ and the last two summands can be replaced with
$-D_\ell^A([x_i\pa_j, g\pa_k])+x_iD_\ell^A(g_j\pa_k)$. To see this we use the explicit calculation of $D_\ell^A([x_i\pa_j, g\pa_k])$.
The difference between left hand side and right hand side of (\ref{der}) reads now
\begin{equation}\label{der3}
f\left(-[x_i\pa_j,D_\ell^A(g\pa_k)]-D_\ell^A(g\pa_k)(x_i)\pa_j+x_iD_\ell^A(g_j\pa_k)\right)\,.
\end{equation}
Note that
$$[x_i\pa_j,D_\ell^A(g\pa_k)]=x_i[\pa_j,D_\ell^A(g\pa_k)]-D_\ell^A(g\pa_k)(x_i)\pa_j$$
and
$$D_\ell^A(g_j\pa_k)=D_\ell^A([\pa_j, g\pa_k])=[\pa_j, D_\ell^A(g\pa_k)]\,.$$
We can therefore transform (\ref{der3}) in the following way
\begin{multline*}
f\left(-[x_i\pa_j,D_\ell^A(g\pa_k)]-D_\ell^A(g\pa_k)(x_i)\pa_j+x_iD_\ell^A(g_j\pa_k)\right)= \\
f\left(-x_i[\pa_j,D_\ell^A(g\pa_k)]+D_\ell^A(g\pa_k)(x_i)\pa_j-D_\ell^A(g\pa_k)(x_i)\pa_j+x_i[\pa_j, D_\ell^A(g\pa_k)]\right)=0\,,
\end{multline*}
which concludes the inductive proof.
\end{proof}

\begin{example}
Take $\R^3$ with coordinates $(x,y,z)$ of weights 1, 2, and 4 respectively. Here ${\mathfrak g}_{<0}(h)$ is spanned by
${\mathfrak g}_{-4}(h)=\la\pa_z\ran$ of weight -4, ${\mathfrak g}_{-3}(h)=\la x\pa_z\ran$ of weight -3, ${\mathfrak g}_{-2}(h)=\la\pa_y,x^2\pa_z, y\pa_z\ran$ ,and ${\mathfrak g}_{-1}(h)=\la\pa_x,x\pa_y,yx\pa_z,x^3\pa_z\ran$ of weight -1. The linear map $D:{\mathfrak g}_{<0}(h)\to{\mathfrak g}_{<0}(h)$ defined by
$$D(y\pa_z)=\pa_y\,\ D(yx\pa_z)=x\pa_y$$ and vanishing on the rest of the basis can be checked to be a derivation. This is derivation of weight 0 which vanishes on coordinate vector fields, so cannot come from a vector field and therefore is of second kind.
\end{example}

\begin{example}
Another example of a dilation $h$ with a derivation of second hand is the one on $\R^3=\{( x_1,x_2,y)\}$
with $x_i$ of degree 1 and $y$ of degree 3.
The derivation of ${\mathfrak g}_{<0}(h)$,
\beas
&D(x_1\pa_y)=\pa_{x_2}\,,\ D(x_2\pa_y)=\pa_{x_1}\\
&D(x_1^2\pa_y)=2x_1\pa_{x_2}\,,\ D(x_2^2\pa_y)=2x_2\pa_{x_1}\,,\ D(x_1x_2\pa_y)=x_1\pa_{x_1}+x_2\pa_{x_2}\,,
\eeas
and the rest of homogeneous basis mapped to 0, is of degree 1 and does not come from a vector field i.e. is of the second kind.
\end{example}

\begin{definition}
We say that $\ell$ is the \emph{first wrong weight} of the dilation $h$ of $\R^n$
if the Tanaka prolongations to weight $\ell-1$ are vector fields:
$${\mathfrak g}^{(k)}_T({\mathfrak m})={\mathfrak m}\oplus{\mathfrak g}_0(h)\oplus{\mathfrak g}_1(h)\oplus...\oplus{\mathfrak g}_k(h)\,,
$$ for $k<\ell$, i.e. any derivation of ${\mathfrak m}$ of weight $k<\ell$ comes from a vector field of weight $k$,
and the Tanaka prolongation ${\mathfrak g}_\ell^T$ is bigger than ${\mathfrak g}_\ell(h)$.
\end{definition}
	
\begin{theorem}\label{vf} Suppose that for the dilation $h$ we have $\ell=r_n-2r_{n-1}$. Then, $\ell$ is the first wrong weight for $h$,
i.e. the space of derivations of weight $\ell$ consists of vector fields of weight $\ell$ extended by derivations of the second kind. In particular, Tanaka prolongation consists of vector fields,
$${\mathfrak g}^{(k)}_T({\mathfrak m})={\mathfrak m}\oplus{\mathfrak g}_0(h)\oplus{\mathfrak g}_1(h)\oplus...\oplus{\mathfrak g}_k(h)\,,
$$
for all $k=0,1,2,\dots,$ if and only if $\ell<0$.
\end{theorem}
\begin{proof}
Suppose that $\ell\geq 0$. We shall prove inductively with respect to the weight of the derivation that ${\mathfrak g}^{(k)}_T({\mathfrak m})$ consists of vector fields only up to $k=\ell-1$. Since the proof of the case of derivation of weight $0$ is essentially the same as the proof of inductive step, we shall proceed with the derivation $D_\zi$ of weight $\zi$ and the assumption that all the derivations of weight less then $\zi$ are given by vector fields. Acting by $D_\zi$ on a vector field of negative weight we obtain a derivation of weight less then $\zi$, i.e. a vector field. We put then $D_\zi(\pa_u)=F_u^v\pa_v$. Since $[\pa_u,\pa_s]=0$ and $D([\pa_u,\pa_s])=[D(\pa_u),\pa_s]+[\pa_u,D(\pa_s)]$, we have
$$[F_u^v\pa_v,\pa_s]+[\pa_u,F_s^w\pa_w]=\left(\frac{\pa F_s^w}{\pa x_u}-\frac{\pa F_u^w}{\pa x_s}\right)\pa_w=0\,.$$
This shows that the 1-forms $\zw^w=F_u^w \xd x^u$ are closed on $\R^n$. Hence, there exist functions $F^w$ such that $F_u^w=\frac{\pa F^w}{\pa x_u}$. If we now put
\be\label{vf}Z=F^w\pa_w\,,\ee
we see that $D_\zi(\pa_u)=[Z,\pa_u]$. We consider then $D'_\zi=D_{\zi}-\ad_Z$ which is a derivation of weight $\zi$ vanishing on coordinate vector fields. Now, it remains to prove that such $D'_\zi$  has to be 0. We shall proceed inductively with respect to the polynomial degree of vector field.

Consider first the linear vector fields $x_a\pa_b$. Since $[\pa_c,x_a\pa_b]=\zd_{ac}\pa_b$, we have
$$0=D'_\zi([\pa_c, x_a\pa_b])=[\pa_c,D'_\zi(x_a\pa_b)]\,.$$
As $D'_\zi(x_a\pa_b)$ is by the inductive assumption a vector field which is annihilated by any coordinate vector field, it is also a coordinate vector field
$$D'_\zi(x_a\pa_b)=d^c_{ab}\pa_c\,.$$
Here, clearly $r_a<r_b$ and $r_c=r_b-r_a-\zi$.

Let us fix $x_a\pa_b$ such that $D'_\zi(x_a\pa_b)\ne 0$, e.g. $d^{s}_{ab}\ne 0$ for some $s$. We will show that $b=n$.  As $r_s=r_b-r_a-\zi<r_b$ and $r_a<r_n$, we have $[x_a\pa_b,x_{s}\pa_n]=0$, and
$$D'_\zi([x_a\pa_b,\,x_{s}\pa_n])=[D'_\zi(x_a\pa_b),\,x_{s}\pa_n]+[x_a\pa_b,\,D'_\zi(x_{s}\pa_n)]=d_{ab}^{s}\pa_n-D'_\zi(x_{s}\pa_n)(x_a)\pa_b$$
which is definitely not zero if $b\neq n$, therefore $D'_\zi(x_a\pa_b)\ne 0$ only for $b=n$. This in particular means also, that there is only one coordinate of highest weight, i.e. $r_{n-1}<r_n$.

Let $i<n$ and put $D'_\zi(x_i\pa_n)=a^k_i\pa_k$. We can think of numbers $a^k_i$ as matrix elements of the real $n\times n$ matrix $A$. As the weight of $D'_\zi(x_i\pa_n)$  is $r_i-r_n+\zi$ which is greater than $-r_n$, we have $a^n_i=0$. It is clear also that $a^n_i=0$ since we act only on vector fields of negative weight. The matrix $A$ then has a form
$$A=\left[\begin{array}{ccc|c}
* & \ldots & * & 0 \\
\vdots &\ddots & \vdots & \vdots\\
* & \ldots & * & 0 \\ \hline
0 & \ldots & 0 & 0
\end{array}\right]$$
Consider the bracket $[x_i\pa_{n},x_j\pa_n]=0$. Applying the derivative we get
\be\label{o}
a^j_i=a^i_j\quad \text{for all} \quad i,j<n\,,
\ee
The matrix $A$ then is symmetric. We will show that $D'_\zi(x_i\pa_n)=0$ for $r_i<r_{n-1}$. For, suppose  $r_i<r_{n-1}$ and consider the bracket $[x_i\pa_{n-1}, x_{n-1}\pa_n]=x_i\pa_n$.
Applying the derivative we get
$$[D'_\zi(x_i\pa_{n-1}),x_{n-1}\pa_n]+[x_i\pa_{n-1},a_{n-1}^k\pa_k]=a_i^j\pa_j\,.$$
But $D'_\zi(x_i\pa_{n-1})=0$ and thus the left hand side equals
$$[x_i\pa_{n-1},a^k_{n-1}\pa_k]=-a^i_{n-1}\pa_{n-1}\,,$$
so what we get is
$$-a^i_{n-1}\pa_{n-1}=a_i^j\pa_j\,.$$
Combining with (\ref{o}) we get $-a_i^{n-1}\pa_{n-1}=a_i^j\pa_j$. This equation means that first $a_i^j=0$ for $j<n-1$ and $i$ such that $r_i<r_{n-1}$ and also that for $j=n-1$ the coefficients $a_i^{n-1}=0$ for $i$ with weight lower than $r_{n-1}$. As a result $D'_\zi(x_i\pa_n)=0$ for $r_i<r_{n-1}$. Moreover by the symmetry condition (\ref{o}) we have $a_i^j=a^i_j$, so $a_{n-1}^i=0$ for $r_i<r_{n-1}$. The matrix $A$ has now the form
$$A=\left[\begin{array}{ccc|c|c}
0 & \ldots & 0 & * & 0 \\
\vdots &\ddots & \vdots & \vdots & \vdots\\
0 & \ldots & 0 & * & 0 \\ \hline
* & \ldots & * & * & 0 \\ \hline
0 & \ldots & 0 & 0 & 0
\end{array}\right]$$
The only nonzero column (and row) represents coefficients corresponding to possibly more than one coordinates with weights equal to $r_{n-1}$. this implies that the only possibility to have $D'_\zi(x_{i}\pa_n)=a^{j}_{i}\p_{j}\neq 0$ is when $r_i=r_j=r_{n-1}$. Calculating the weight of both sides we get
$$r_{n-1}-r_n+\zi=-r_{n-1}\,,$$
and $r_n=2r_{n-1}+\zi$. This contradicts our assumption that $\zi<\ell$. We conclude that the derivative $D'_\zi$ of weight $\zi<\ell$ that vanishes on constant vector fields vanishes also for linear vector fields.

What we will show now is that, for $\zi <\ell$, if derivative $D'_\zi$ of weight $\zi$ of $\mathfrak m$ vanishes on coordinate and linear vector fields, then it vanishes totally. First we prove inductively that if $D'_\zi$ vanishes on vector fields of degree $q$ than it takes constant values for vector fields of degree $q+1$. Indeed, if $g$ is a polynomial of degree $q+1$ then $[g\pa_i,\pa_j]$ is a vector field of degree $q$ then
$$0=D'_\zi([g\pa_i,\pa_j])=[D'_\zi(g\pa_i),\pa_j]\,,$$
so the $D'_\zi(g\pa_i)$ commutes with all coordinate vector filed hence it is constant vector field, by the inductive assumption.

Now take $f\pa_j$ with $f$ being a homogeneous polynomial which is at least quadratic and such that $D'_\zi(f\pa_j)\ne 0$.
We have then $D'_\zi(f\pa_j)=\pa_c$. Consider now $[f\pa_j,x_c\pa_j]$. Note that since $r_c=-r_f+r_j-\zi$ we have $\sw(x_c\pa_j)=r_c-r_j=-r_f-\zi$, so that $x_c\pa_j\in{\mathfrak m}$. We get $[f\pa_j,x_c\pa_j]=0$, since $r_c<r_j$ and $r_j>r_f$. Applying $D'_\zi$, we get
$$0=D'_\zi([f\pa_j,x_c\pa_j])=[D'_\zi(f\pa_j),x_c\pa_j]=[\pa_c,x_c\pa_j]=\pa_j\,,$$
which is a contradiction. Hence, any derivation of weight $\zi$ vanishing on constant vector fields vanishes globally, and the vector field $(\ref{vf})$ is the `inner' derivation we were looking for.

Finally, from the proof is clear that if $\zi=\ell$, then any derivations of degree $\ell$ which is not a vector field is a derivation of the second kind, an that derivations of the second kind of weight $\ell$ exist.
\end{proof}
\begin{corollary}
Let $h$ be a dilation. Derivations of weight 0 of the nilpotent Lie algebra ${\mathfrak g}_{< 0}(h)$ are determined by the adjoint action of vector fields from ${\mathfrak g}_{0}(h)$ if $r_n-2r_{n-1}\ne0$ and consist of ${\mathfrak g}_{0}(h)$ extended by derivations of the second kind if $r_n=2r_{n-1}$.

\end{corollary}
\section{Concluding remarks}
We found a necessary and sufficient condition for the signature $(r_1,\dots,r_n$ of a dilation $h$ on $\R^n$ assuring that all derivations of weight 0 (more generally, the Tanaka prolongation of the order $k$) of the nilpotent Lie algebra ${\mathfrak g}_{<0}(h)$ of negative vector fields with respect to $h$ come from the Lie algebra ${\mathfrak g}_{0}(h)$ of vector fields of weight 0 (come from ${\mathfrak g}_{\le k}(h)$). Surprisingly enough, there are `strange' derivations for some $h$ which we described in detail.


\small{\vskip1cm

\noindent Katarzyna GRABOWSKA\\
Faculty of Physics\\
                University of Warsaw} \\
               Pasteura 5, 02-093 Warszawa, Poland
                 \\Email: konieczn@fuw.edu.pl \\

\noindent Janusz GRABOWSKI\\ Institute of
Mathematics\\  Polish Academy of Sciences\\ \'Sniadeckich 8, 00-656 Warszawa, Poland
\\Email: jagrab@impan.pl \\

\noindent Zohreh RAVANPAK\\ Institute of
Mathematics\\  Polish Academy of Sciences\\ \'Sniadeckich 8, 00-656 Warszawa, Poland
\\Email: zravanpak@impan.pl \\

\end{document}